\documentclass[12pt]{amsart}
\advance\hoffset by -0.5 truecm
\usepackage{amssymb}
\newtheorem{Theorem}{Theorem}[section]
\newtheorem{Lemma}[Theorem]{Lemma}
\newtheorem{Corollary}[Theorem]{Corollary}

\newtheorem{Proposition}[Theorem]{Proposition}

\newtheorem{Definition}[Theorem]{Definition}
\newtheorem{Remark}[Theorem]{Remark}

\def \dim{{\mbox {dim}}\,}

\def\V{\mbox{Var}}

\def\Z{{\mathbb Z}}
\def\R\re
\def\V{\bf V}

\def \re{{\mathbb R}}

\def \C{{\mathbb C}}

\def \V{{\bf V}}

\def \M{{\mathbb M}}

\def \cA{{\mathcal A}}
\def \cH{{\mathcal H}}

\def \g{\mathfrak{g}}
\def \P{\mathcal P}
\def \s{\mathfrak{so}(3)}

\begin{document}
\title[Transparent pairs]{Transparent pairs}

\author[G.P. Paternain]{Gabriel P. Paternain}
 \address{ Department of Pure Mathematics and Mathematical Statistics,
University of Cambridge,
Cambridge CB3 0WB, UK}
 \email {g.p.paternain@dpmms.cam.ac.uk}




\begin{abstract} Let $M$ be a closed orientable Riemannian surface. Consider an $SO(3)$-connection $A$ and a Higgs field $\Phi:M\to\s$.
The pair $(A,\Phi)$ naturally induces a cocycle over the geodesic flow of $M$.
We classify (up to gauge transformations) cohomologically trivial pairs $(A,\Phi)$ with finite Fourier series
in terms of a suitable B\"acklund transformation. In particular, if $M$ is negatively curved we obtain a full classification of $SO(3)$-transparent pairs.

\end{abstract}

\maketitle

\section{Introduction}

Let $(M,g)$ be a closed oriented Riemannian surface and let $A$ be a $\g$-valued 1-form, where $\g$ is $\mathfrak{su}(n)$ or $\mathfrak{so}(n)$. We think of $A$ as a smooth map $A:TM\to\g$ which is linear in $v\in T_{x}M$ for all $x\in M$. The 1-form $A$ defines a unitary/orthogonal connection $d+A$ on the trivial bundle $M\times\mathbb{F}^n$, where $\mathbb{F}=\re,\C$. Suppose we are given in addition
a smooth map $\Phi:M\to\g$ which we call the {\it Higgs field}.

The pair $(A,\Phi)$ naturally induces a cocycle over
the geodesic flow $\phi_t$ of the metric $g$ acting on the unit sphere bundle
$SM$ with projection $\pi:SM\to M$. The cocycle takes values in the group $G=SU(n),SO(n)$ and
is defined as follows: let
$C:SM\times \re\to G$ be determined by
\[\frac{d}{dt}C(x,v,t)=-(A(\phi_{t}(x,v))+\Phi(\pi\circ\phi_{t}(x,v)))C(x,v,t),\;\;\;\;\;C(x,v,0)=\mbox{\rm Id}.\]
The function $C$ is a {\it cocycle}: 
\[C(x,v,t+s)=C(\phi_{t}(x,v),s)\,C(x,v,t)\]
for all $(x,v)\in SM$ and $s,t\in\re$. The cocycle $C$ is said to be {\it cohomologically trivial}
if there exists a smooth function $u:SM\to G$ such that
\[C(x,v,t)=u(\phi_{t}(x,v))u^{-1}(x,v)\]
for all $(x,v)\in SM$ and $t\in\re$. We call $u$ a trivializing function and note that two trivializing functions $u_{1}$ and $u_{2}$ (for the same cocycle) are related by $u_{2}w=u_{1}$
where $w:SM\to G$ is constant along the orbits of the geodesic flow. In particular, if $\phi_{t}$ is transitive (i.e. there is a dense orbit) there is a unique trivializing function up to
right multiplication by a constant matrix in $G$.

\begin{Definition} {\rm We will say that a pair $(A,\Phi)$ is cohomologically trivial if $C$ is cohomologically trivial.}
\end{Definition}

Observe that the gauge group given by the set of smooth maps
$r:M\to G$ acts on pairs as follows:
\[(A,\Phi)\mapsto (r^{-1}dr+r^{-1}Ar,r^{-1}\Phi r).\]
This action leaves invariant the set of cohomologically trivial pairs: indeed,
if $u$ trivializes the cocycle $C$ of a pair $(A,\Phi)$, then it is easy to check that $r^{-1}u$ trivializes the cocycle of the pair $(r^{-1}dr+r^{-1}Ar,r^{-1}\Phi r)$.

In the present paper we will classify up to gauge equivalence all cohomologically trivial $SO(3)$-pairs
which admit a trivializing function with {\it finite} Fourier series.
By finite Fourier series we mean the following. Consider a smooth map
$u:SM\to \M_{n}(\mathbb C)$, where $\M_{n}(\mathbb C)$ is the set of $n\times n$ complex matrices (we think of maps into $SO(3)$ also as maps into $\M_{3}(\C)$). Let $\rho_{t}$ be the flow of the vertical vector field $V$ determined by the principal circle fibration $\pi:SM\to M$. We shall say that $u$ has finite Fourier series
if there exists a non-negative integer $N$ such that
\[u_{m}(x,v):=\frac{1}{2\pi}\int_{0}^{2\pi}u(\rho_{t}(x,v))e^{-imt}\,dt\]
vanishes identically for all $m$ with $|m|\geq N+1$. This means that
we can expand $u$ as a finite sum $u=\sum_{m=-N}^{m=N}u_{m}$.

The classification of cohomologically trivial $SO(3)$-pairs with finite Fourier series comes in terms of a suitable B\"acklund transformation. Basically this means that there is an explicit algebraic way of constructing new cohomologically trivial pairs from known ones and it works as follows. Given a cohomologically trivial pair $(A,\Phi)$ with a trivializing function $b$ we seek a smooth
map $a:SM\to SO(3)$ such that $u:=ab$ is a trivializing function for the new pair.
The map $a$ will be special, of degree one in the velocities and such that $g:=a^{-1}V(a)$ is a map from $M$ into $\s$ with norm one, i.e. taking values
in $S^2\subset \s$.
We will show that $u$ is a trivializing function of a new pair if $g$ satisfies the equation $-\star d_{A}g=[d_{A}g,g]$, where $\star$ is the Hodge star operator of the metric and $d_{A}$ is the covariant derivative on endomorphisms associated with the connection $A$. This will form the basis of our classification and for a detailed statement we refer to Theorem \ref{cor:back} below. 
A similar procedure was employed in \cite{P1} to classify cohomologically trivial $SU(2)$-connections. However if one wishes to include a Higgs field $\Phi$
the method in \cite{P1} has to be modified. Perhaps, the main contribution of this paper
is the observation that in order to consider Higgs fields (even in the $SU(2)$ case) it is better to work
with $SO(3)$ rather $SU(2)$. Roughly speaking, the $SO(3)$ B\"acklund transformation of this paper acts as the ``square root'' of the $SU(2)$ B\"acklund transformation from \cite{P1}, this is explained in detail in Section \ref{bt}.
Note that a classification of $SO(3)$ cohomologically trivial pairs includes a classification of $SU(2)$ pairs, since any $SU(2)$ cocycle can be converted into
an $SO(3)$ cocycle just by composing with the 2-1 homomorphism $SU(2)\to SO(3)$.

The motivation for understanding cohomologically trivial pairs comes from
the related weaker notion of {\it transparent pairs}. The pair $(A,\Phi)$
is said to be transparent if $C(x,v,T)=\mbox{\rm Id}$ every time
that $\phi_{T}(x,v)=(x,v)$. In other words the linear transport equation associated
with $A+\Phi$ has trivial holonomy along closed geodesics\footnote{To avoid cumbersome notation we will often write $A+\Phi$ instead of $A+\Phi\circ\pi$.}.
These pairs are invisible from the point of view of the closed geodesics of the Riemannian metric
and it is a natural inverse problem to try to determine them.
This inverse problem can also be considered in the case of manifolds with boundary or $\re^d$ with appropriate decay
conditions at infinity, and in this context, it has been addressed by several authors, see for instance \cite{E,FU,No,Sha,V}. The inclusion of a Higgs field is natural and transparent pairs on $S^2$ are related to solutions of the
Bogomolny equations $D\Phi=\star F$ in $(2+1)$-dimensional Minkowski space \cite 
{Wa1,Wa2}. The inverse problem with a Higgs field is also related with the attenuated
X-ray transform as explained in \cite[Section 8]{No}.

Obviously a cohomologically trivial pair is transparent.
There is one important situation in which both notions agree. If $\phi_t$ is
Anosov, then the Livsic theorem \cite{L1,L2} together with the regularity results in \cite{NT}
imply that a transparent pair is also cohomologically trivial. 
The Anosov property is satisfied, if for example $(M,g)$ has negative curvature.

In Theorem \ref{theorem:finite} we will show that if $(M,g)$ is negatively 
curved, then the trivializing function of a cohomologically trivial pair (unique up to multiplication on the right by a constant)
must have finite Fourier series. We do not know if this result holds true just assuming that the geodesic flow is Anosov.

In summary we end up proving (see Theorem \ref{theorem:main}):

\medskip

\noindent{\bf Theorem.} {\it Let $M$ be a closed orientable surface of negative curvature.
Then any transparent $SO(3)$-pair can be obtained by successive applications
of B\"acklund transformations as described in Theorem \ref{cor:back}.}

\medskip

Our results also give a classification of transparent $SO(3)$ pairs
with finite Fourier series over an arbitrary Zoll surface (i.e. a Riemannian 
metric on $S^2$ all of whose geodesics are closed). This is because
we show that in this case (cf. Corollary \ref{cor:zoll}) any transparent SO(3) pair is cohomologically trivial. Of course in the Zoll case there may be many other cohomologically trivial pairs which do not have a finite Fourier series. A full classification of $U(n)$ transparent connections ($\Phi=0$) for the case of the round metric on $S^2$ has been obtained by L. Mason (unpublished) using methods from twistor theory as in \cite{Ma}.

\medskip

\noindent{\it Acknowledgements:} I am very grateful to the MSRI and to the organizers of the program Inverse Problems and Applications for hospitality while this work was being completed. I am also very grateful to Maciej Dunajski for several useful conversations related to this paper and to Will Merry for comments and corrections to a previous draft.

\section{Cocycles and auxiliary results}

Let $N$ be a closed manifold and $\phi_t:N\to N$ a smooth
flow with infinitesimal generator $X$.

Let $G$ be a compact Lie group; for the purposes
of this paper it is enough to think of $G$ as $SU(n)$ or $SO(3)$.

\begin{Definition} A $G$-valued cocycle over the flow $\phi_t$ is a map
$C:N\times\re\to G$ that satisfies
\[C(x,t+s)=C(\phi_{t}x,s)\,C(x,t)\]
for all $x\in N$ and $s,t\in\re$.
\label{def1}
\end{Definition}

In this paper the cocycles will always be smooth. In this case $C$ is
determined by its infinitesimal generator $B:N\to\mathfrak g$ given by
\[B(x):=-\left.\frac{d}{dt}\right|_{t=0}C(x,t).\]
The cocycle can be recovered from $B$ as the unique solution to
\[\frac{d}{dt}C(x,t)=-dR_{C(x,t)}(B(\phi_{t}x)),\;\;\;C(x,0)=\mbox{\rm Id},\]
where $R_{g}$ is right translation by $g\in G$.
We will indistinctly use the word ``cocycle'' for $C$ or its infinitesimal generator $B$.

\begin{Definition} The cocycle $C$ is said to be {\it cohomologically trivial}
if there exists a smooth function $u:N\to G$ such that
\[C(x,t)=u(\phi_{t}x)u^{-1}(x)\]
for all $x\in N$ and $t\in\re$.
\end{Definition}

Observe that the condition of being cohomologically trivial can be equivalently expressed in terms of the infinitesimal generator $B$ of the cocycle by saying that there exists a smooth function $u:N\to G$ that satisfies the equation
\[d_{x}u(X(x))+d_{Id}R_{u(x)}(B(x))=0\]
for all $x\in N$. If $G=SU(n)$ or $SO(3)$ we can write this more succinctly
as
\[X(u)+Bu=0\]
where it is understood that differentiation and multiplication
is in the set of complex $n\times n$-matrices for $SU(n)$ or
in the set of real $3\times 3$-matrices for $SO(3)$.

\begin{Definition} A cocycle $C$ is said to satisfy the periodic orbit
obstruction condition if $C(x,T)=\mbox{\rm Id}$ whenever
$\phi_{T}x=x$.
\end{Definition}

Obviously a cohomologically trivial cocycle satisfies the periodic
orbit obstruction condition. The converse turns out to be true
for transitive Anosov flows: this is one of the celebrated
Livsic theorems \cite{L1,L2,NT}.

\begin{Theorem}[The smooth Livsic periodic data theorem] Suppose
$\phi_t$ is a smooth transitive Anosov flow.
Let $C$ be a smooth cocycle such that $C(x,T)=\mbox{\rm Id}$ whenever
$\phi_{T}x=x$. Then $C$ is cohomologically trivial.
\label{livsic}
\end{Theorem}

There is another situation of interest to us in which a cocycle that satisfies
the periodic orbit obstruction condition is cohomologically
trivial. Suppose $N$ is a closed orientable 3-manifold which admits
a free circle action with infinitesimal generator $V$ and flow $\phi_t$.
The quotient of $N$ by the circle action is a closed orientable surface
$M$; let $e(N)\in H^{2}(M,\Z)=\Z$ be the Euler class of the circle bundle $p:N\to M$.

\begin{Proposition} Any cocycle $B:N\to \mathfrak{su}(n)$ over $\phi_{t}$ which satisfies the
periodic orbit obstruction condition is cohomologically trivial.
The same holds for any cocycle $B:N\to \mathfrak{so}(3)$ provided
that $e(N)$ is even.
\label{prop:trivial}
\end{Proposition}

\begin{proof} Consider first the case of $B:N\to \mathfrak{su}(n)$.
Let $U$ be a neighbourhood of $M$ such that $p^{-1}(U)$ is trivialized as $U\times S^1$ via $\psi_{U}:p^{-1}(U)\to U\times S^1$. In this trivialization we may write $V=\partial/\partial\theta$, where
$\theta\in \re/2\pi\Z=S^1$. The fact that $B$ satisfies the periodic orbit obstruction condition means that we can find a smooth $u_{U}:U\times S^1\to SU(n)$
such that 
\begin{equation}
\frac{\partial u_{U}}{\partial\theta}(x,\theta)+B(x,\theta)u_{U}(x,\theta)=0
\label{eq:co}
\end{equation}
for all $(x,\theta)\in U\times S^1$ and $u_{U}(x,0)=\mbox{\rm Id}$.
We will show that the functions $u_{U}$ can be glued to determine a globally
defined $u:N\to SU(n)$ which solves $V(u)+Bu=0$.

Observe first that if $g:U\to SU(n)$ is any smooth map, then $u_{U}g$ also satisfies (\ref{eq:co}).

Take now another set $U'$ which overlaps with $U$ and over which we can
 trivialize $p:N\to M$. We obtain a transition function 
$\psi_{UU'}:U\cap U'\to S^1$. To have a globally defined $u$ we need the existence of smooth functions
$g_{U}:U\to SU(n)$ such that
\[u_{U}(x,\psi_{UU'}(x))g_{U}(x)=g_{U'}(x)\]
for all $x\in U\cap U'$ and all overlaps of a covering of $M$ by trivializing
sets. Observe that by construction
\begin{equation}
u_{U'}(x,\theta)=u_{U}(x,\theta+\psi_{UU'}(x))[u_{U}(x,\psi_{UU'}(x))]^{-1}\;\;\;\;\mbox{\rm on}\;(U\cap U')\times S^1.
\label{eq:agree}
\end{equation}

The key observation now is that $\varphi_{UU'}:U\cap U'\to SU(n)$ given
by
\[\varphi_{UU'}(x):=u_{U}(x,\psi_{UU'}(x))\]
is an $SU(n)$-cocycle in the sense of principal bundles. Indeed, the cocycle property
$\varphi_{UU"}(x)=\varphi_{U'U''}(x)\,\varphi_{UU'}(x)$ follows right away from the fact
that $\psi_{UU'}$ is an $S^1$-cocycle and (\ref{eq:agree}). But any principal $SU(n)$-bundle over a surface is trivial, thus there exist smooth functions
$g_{U}:U\to SU(n)$ such that
\[\varphi_{UU'}(x)=g_{U'}(x)\,(g_{U}(x))^{-1}\]
which says precisely that $u$ can be defined globally.

The argument for $SO(3)$ is pretty much the same except that we now have
two different $SO(3)$-bundles over the surface, the non-trivial one having
non-zero Stiefel-Whitney class $w_2\in H^{2}(M,\Z_{2})$.
Let $E$ denote the principal bundle determined by $\varphi_{UU'}$ and
let us check that $E$ is trivial if $e(N)$ is even. Perhaps the easiest way to see this is to argue as follows. Consider the pull-back bundle $p^*E$ and observe
that it is trivial. Indeed, its transition functions are given by
\[\varphi^*_{p^{-1}(U)p^{-1}(U')}(y)=\varphi_{UU'}(py)\]
and if we set
\[g^{*}_{p^{-1}(U)}(y)=[u_{U}\circ\psi_{U}(y)]^{-1}\]
then
\[\varphi^*_{p^{-1}(U)p^{-1}(U')}(y)=g^{*}_{p^{-1}(U')}\,\left(g^{*}_{p^{-1}(U)}\right)^{-1}.\]
It follows that $w_{2}(p^{*}E)=p^{*}w_{2}(E)=0$. But using the Gysin sequence
of the circle bundle $p:N\to M$, we see that $p^{*}w_{2}(E)$ vanishes iff
$w_{2}(E)$ is in the image of the map $H^{0}(M,\Z_{2})\to H^{2}(M,\Z_{2})$ given by the cup product with the Euler class $e(N)$.
Hence $w_{2}(E)=0$ as desired.

\end{proof}

\begin{Remark}{\rm So far we have considered {\it right} cocycles. However
one can easily go from a right cocycle $C$ to a left cocycle simply by
considering $C^{-1}$ and thus Proposition \ref{prop:trivial} applies to them as well.

}
\end{Remark}

Let us now describe the two main consequences of this result.

\subsection{Zoll surfaces} Recall that a Zoll surface is a Riemannian
metric on $S^2$ all of whose geodesics are closed. By \cite{GG} all the closed geodesics are simple and with the same length and thus the action of the geodesic flow on the unit circle bundle $SM$ fits the setting above. In fact, by \cite{We} the Euler class of the circle bundle determined by the action of the geodesic flow is even.

Observe that by definition, saying that a pair $(A,\Phi)$ is transparent is the same as saying that the cocycle determined
by $(A,\Phi)$ satisfies the periodic orbit obstruction condition. Thus Proposition \ref{prop:trivial} gives right away the following:

\begin{Corollary} Let $(S^2,g)$ be a Zoll surface. Then any transparent $SU(n)$ or SO(3)-pair $(A,\Phi)$ is cohomologically trivial.
\label{cor:zoll}
\end {Corollary}

\subsection{An important step to set up the $SO(3)$ B\"acklund transformation}

In the next application of Proposition \ref{prop:trivial}, our circle
bundle is the unit tangent bundle $\pi:SM\to M$ and the circle action
is given by the flow of the vertical vector field $V$. Since $M$ is 
orientable, 
the Euler class of this bundle is obviously even (the Euler characteristic 
of $M$ is even). We consider the left cocycle over $V$ determined by 
$f:M\to \mathfrak{su}(n)$ or $f:M\to \mathfrak{so}(3)$ which we also regard
as a function over $SM$ by composing it with $\pi$. In other words, if $\rho_t$ is the flow of $V$, we consider the cocycle determined by
\[\partial_{t}C(x,v,t)=C(x,v,t)f(\pi\circ\rho_{t}(x,v)),\;\;\;C(x,v,0)=\mbox{\rm Id}.\]
Since $f$ only depends on $x\in M$ we can write $C(x,v,t)=\exp(t\,f(x))$.

There are two situations of interest to us. The first is $f:M\to\mathfrak{su}(2)$ with $f^{2}=-\mbox{\rm Id}$. Under this condition for any $x\in M$, $f(x)$ has eigenvalues $\pm i$ and it clearly satisfies the periodic orbit obstruction condition with respect to $V$, thus we have:

\begin{Corollary} Let $f:M\to\mathfrak{su}(2)$ be a smooth map
with $f^{2}=-\mbox{\rm Id}$. Then there exists a smooth $u:SM\to SU(2)$
such that $uf=V(u)$.
\end{Corollary}

This corollary was proved with slightly different methods 
in \cite[Lemma 4.1]{P1}. Observe that any two solutions $u$ and $v$ of $uf=V(u)$
are related by $u=rv$, where $r:M\to SU(2)$. 
Consider local coordinates $(x,y)$ on a neighbourhood $U$ of $M$. This gives coordinates $(x,y,\theta)$ on $SM$ where
$\theta$ is the angle between a unit vector $v$ and $\partial/\partial x$.
In these coordinates, $V=\partial/\partial\theta$ and any solution $u$ of
$uf=V(u)$ can be expressed locally as
\[u(x,y,\theta)=r(x,y)(\cos\theta\,\mbox{\rm Id}+\sin\theta\,f(x,y)),\]
for some smooth $r:U\to SU(2)$.

The next situation of interest to us is that of a smooth $f:M\to\mathfrak{so}(3)$ with $f^{3}+f=0$. Again, under this condition $f$ defines an $SO(3)$-cocycle over $V$ satisfying the periodic orbit obstruction condition and thus
by Proposition \ref{prop:trivial} we have:

\begin{Corollary} Let $f:M\to\mathfrak{so}(3)$ be a smooth map
with $f^{3}+f=0$. Then there exists a smooth $u:SM\to SO(3)$
such that $uf=V(u)$.
\label{cor:existencea}
\end{Corollary}

This last corollary will be essential to run the B\"acklund transformation
for transparent pairs.

As above, if we consider local coordinates any solution $u$ to $uf=V(u)$ can be written locally as
\[u(x,y,\theta)=r(x,y)(\mbox{\rm Id}+f(x,y)^2+\sin\theta\,f(x,y)-\cos\theta\,f^{2}(x,y))\]
for some smooth $r:U\to SO(3)$. Note that
\[\mbox{\rm Id}+f^2+\sin\theta\,f-\cos\theta\,f^{2}=
-\frac{e^{-i\theta}}{2}f(f-i\mbox{\rm Id})+(\mbox{\rm Id}+f^2)-\frac{e^{i\theta}}{2}f(f+i\mbox{\rm Id}).\]
From these local expressions it is clear that $u=u_{-1}+u_{0}+u_{1}$ with
\begin{align*}
&\mbox{\rm Ker}\,u_{0}=E_{i}\oplus E_{-i},\\
&\mbox{\rm Ker}\,u_{1}=\mbox{\rm Ker}(f)\oplus E_{-i},\\
&\mbox{\rm Ker}\,u_{-1}=\mbox{\rm Ker}(f)\oplus E_{i},
\end{align*}
where $E_{\pm i}$ is the eigenspace of $f$ corresponding to the eigenvalue $\pm i$.

\section{The set up}
\label{prelim}
In this section we summarise the setting in \cite{P} and the results needed for
the subsequent sections.

Let $M$ be an oriented surface with a Riemannian metric and let
$SM$ be its unit tangent bundle. Recall that $SM$ has a
canonical framing $\{X,H,V\}$, where $X$ is the geodesic vector field, $V$
is the vertical vector field and $H=[V,X]$ is the horizontal vector field.

Let $\M_{n}(\mathbb C)$ be the set of $n\times n$ complex
matrices. Given functions $u,v:SM\to \M_{n}(\mathbb C)$ we consider the
inner product
\[\langle u,v \rangle =\int_{SM}\mbox{\rm trace}\,(u\,v^*)\,d\mu,\]
where $\mu$ is the Riemannian measure associated with the Sasaki metric
of $SM$ which makes $\{X,H,V\}$ into an orthonormal frame.
The space $L^{2}(SM,\M_{n}(\mathbb C))$ decomposes orthogonally
as a direct sum
\[L^{2}(SM,\M_{n}(\mathbb C))=\bigoplus_{m\in\mathbb Z}H_{m}\]
where $-iV$ acts as $m\,\mbox{\rm Id}$ on $H_m$.

Following Guillemin and Kazhdan in \cite{GK} we introduce the following
first order elliptic operators 
$$\eta_{+},\eta_{-}:C^{\infty}(SM,\M_{n}(\mathbb C))\to
C^{\infty}(SM,\M_{n}(\C))$$ given by
\[\eta_{+}:=(X-iH)/2,\;\;\;\;\;\;\eta_{-}:=(X+iH)/2.\]
Clearly $X=\eta_{+}+\eta_{-}$. Let $\Omega_{m}:=C^{\infty}(SM,\M_{n}(\C))\cap H_{m}$.
We have
\[\eta_{+}:\Omega_{m}\to \Omega_{m+1},\;\;\;\;\eta_{-}:\Omega_{m}\to \Omega_{m-1},\;\;\;\;(\eta_{+})^{*}=-\eta_{-}.\]

If a $U(n)$ pair $(A,\Phi)$ is cohomologically trivial there exists a smooth $u:SM\to U(n)$ such that $C(x,v,t)=u(\phi_{t}(x,v))u^{-1}(x,v)$.
Differentiating with respect to $t$ and setting $t=0$, this is equivalent to
$X(u)+(A+\Phi)u=0$, where we now regard $A$ and $\Phi$ as functions $A,\Phi:SM\to\mathfrak{u}(n)$.
To deal with this equation, we introduce the ``twisted'' operators
\[\mu_{+}:=\eta_{+}+A_1,\;\;\;\;\mu_{-}:=\eta_{-}+A_{-1}.\]
where $A=A_{-1}+A_{1}$, and
\[A_{1}:=\frac{A-iV(A)}{2}\in H_{1},\]
\[A_{-1}:=\frac{A+iV(A)}{2}\in H_{-1}.\]
Observe that this decomposition corresponds precisely with the usual decomposition of $\mathfrak{u}(n)$-valued 1-forms
on a surface:
\[\Omega^{1}(M,\mathfrak{u}(n))\otimes \C=\Omega^{1,0}(M,\mathfrak{u}(n))\oplus \Omega^{0,1}(M,\mathfrak{u}(n)),\]
where $\star=-i$ on $\Omega^{1,0}$ and $\star=i$ on $\Omega^{0,1}$ (here $\star$ is the Hodge star operator of the metric).

We also have
\[\mu_{+}:\Omega_{m}\to \Omega_{m+1},\;\;\;\;\mu_{-}:\Omega_{m}\to \Omega_{m-1},\;\;\;\;(\mu_{+})^{*}=-\mu_{-}.\]
The equation $X(u)+(A+\Phi)u=0$ is now $\mu_{+}(u)+\mu_{-}(u)+\Phi u=0$.

For future use, it is convenient to write the operators $\eta_{-}$ and $\mu_{-}$ in local
coordinates. Consider isothermal coordinates $(x,y)$ on $M$ such that the metric
can be written as $ds^2=e^{2\lambda}(dx^2+dy^2)$ where $\lambda$ is a smooth
real-valued function of $(x,y)$. This gives coordinates $(x,y,\theta)$ on $SM$ where
$\theta$ is the angle between a unit vector $v$ and $\partial/\partial x$.
In these coordinates, $V=\partial/\partial\theta$ and
the vector fields $X$ and $H$ are given by:
\[X=e^{-\lambda}\left(\cos\theta\frac{\partial}{\partial x}+
\sin\theta\frac{\partial}{\partial y}+
\left(-\frac{\partial \lambda}{\partial x}\sin\theta+\frac{\partial\lambda}{\partial y}\cos\theta\right)\frac{\partial}{\partial \theta}\right);\]
\[H=e^{-\lambda}\left(-\sin\theta\frac{\partial}{\partial x}+
\cos\theta\frac{\partial}{\partial y}-
\left(\frac{\partial \lambda}{\partial x}\cos\theta+\frac{\partial \lambda}{\partial y}\sin\theta\right)\frac{\partial}{\partial \theta}\right).\]
Consider $u\in\Omega_m$ and write it locally as $u(x,y,\theta)=h(x,y)e^{im\theta}$.
Using these formulas a simple, but tedious calculation shows that
\begin{equation}
\eta_{-}(u)=e^{-(1+m)\lambda}\bar{\partial}(he^{m\lambda})e^{i(m-1)\theta},
\label{eq:eta}
\end{equation}
where $\bar{\partial}=\frac{1}{2}\left(\frac{\partial}{\partial x}+i\frac{\partial}{\partial y}\right)$.
In order to write $\mu_{-}$ suppose that $A(x,y,\theta)=a(x,y)\cos\theta+b(x,y)\sin\theta$.
If we also write $A=A_{x}dx+A_{y}dy$, then $A_x=ae^{\lambda}$ and $A_y=be^{\lambda}$.
Let $A_{\bar{z}}:=\frac{1}{2}(A_{x}+iA_{y})$.
Using the definition of $A_{-1}$ we derive
\begin{equation}
A_{-1}=\frac{1}{2}(a+ib)e^{-i\theta}=A_{\bar{z}}d\bar{z}.
\label{eq:a1}
\end{equation}
Putting this together with (\ref{eq:eta}) we obtain
\begin{equation}
\mu_{-}(u)=e^{-(1+m)\lambda}\left(\bar{\partial}(he^{m\lambda})+A_{\bar{z}}he^{m\lambda}\right)e^{i(m-1)\theta}.
\label{eq:mu}
\end{equation}

Note that $\Omega_m$ can be identified with the set of smooth sections
of the bundle $(M\times \M_{n}(\mathbb C))\otimes K^{\otimes m}$ where $K$ is the canonical line bundle. The identification takes $u=he^{im\theta}$ into $he^{m\lambda}(dz)^m$ ($m\geq 0$)
and $u=he^{-im\theta}\in \Omega_{-m}$ into $he^{m\lambda}(d\bar{z})^m$.
The second equality in (\ref{eq:a1}) should be understood using this
identification.

\section{Finite Fourier series in negative curvature}

Given an element $u\in C^{\infty}(SM,\M_{n}(\C))$, we write
$u=\sum_{m\in\Z}u_{m}$, where $u_m\in \Omega_m$.
We will say that $u$ has degree $N$, if $N$ is the smallest
non-negative integer such that $u_{m}=0$ for all $m$ with $|m|\geq N+1$.
The following finiteness result will be important for us.

\begin{Theorem} If $M$ has negative curvature every solution
$u$ of $X(u)+(A+\Phi)u=0$ has finite degree.
\label{theorem:finite}
\end{Theorem}

This theorem was proved in \cite[Theorem 5.1]{P} for the case $\Phi=0$
and the extension to include a Higgs field is not entirely straightforward
due to the fact that the commanding recurrence relation has a more complicated nature.

\begin{proof} We shall use the following equality proved in \cite[Corollary 4.4]{P}. Given $u\in C^{\infty}(SM,\M_{n}(\C))$ we have
\[|\mu_{+}u|^2=|\mu_{-}u|^2+\frac{i}{2}(\langle K\,V(u),u\rangle+\langle (\star F_{A})u,u\rangle),\]
where $K$ is the Gaussian curvature of the metric and $F_{A}$ is the curvature of $A$.
Hence for $u_m\in \Omega_{m}$ we have
\[|\mu_{+}u_{m}|^2=|\mu_{-}u_{m}|^2+\frac{1}{2}(\langle (i\star F_{A}-mK\,\mbox{\rm Id})u,u\rangle).\] 
Hence if $K<0$, there exist constants $c_{m}>0$ with $c_{m}\to\infty$
and a positive integer $\ell$ such that
\begin{equation}
|\mu_{+}u_{m}|^2\geq|\mu_{-}u_{m}|^2+c_{m}|u_{m}|^2
\label{eq:exp}
\end{equation}
for all $m\geq \ell$.
We know that for all $m\in \Z$
\begin{equation}
\mu_{+}(u_{m-1})+\mu_{-}(u_{m+1})+\Phi\,u_{m}=0.
\label{eq:recurrence}
\end{equation}
Using (\ref{eq:exp}) and (\ref{eq:recurrence}) we can write ($m\geq \ell-1$)
\begin{align*}
|\mu_{+}(u_{m+1})|^2&\geq |\mu_{-}(u_{m+1})|^2+c_{m+1}|u_{m+1}|^{2}\\
&=|\mu_{+}(u_{m-1})+\Phi\,u_{m}|^{2}+c_{m+1}|u_{m+1}|^2\\
&=|\mu_{+}(u_{m-1})|^{2}+2\Re\langle\mu_{+}(u_{m-1}),\Phi\,u_{m}\rangle+|\Phi\,u_{m}|^{2}+c_{m+1}|u_{m+1}|^2\\
&=|\mu_{+}(u_{m-1})|^{2}-2\Re\langle u_{m-1},\mu_{-}(\Phi\,u_{m})\rangle+|\Phi\,u_{m}|^{2}+c_{m+1}|u_{m+1}|^2.
\end{align*}
Let us compute $\mu_{-}(\Phi\,u_{m})$:
\begin{align*}
\mu_{-}(\Phi\,u_{m})&=\eta_{-}(\Phi\,u_{m})+A_{-1}\Phi\,u_{m}\\
&=\eta_{-}(\Phi)u_{m}+\Phi\eta_{-}(u_{m})+A_{-1}\Phi\,u_{m}\\
&=\Phi\mu_{-}(u_{m})+\bar{\partial}_{A}(\Phi)u_{m}\\
&=\Phi(-\mu_{+}(u_{m-2})-\Phi\,u_{m-1})+\bar{\partial}_{A}(\Phi)u_{m},
\end{align*}
where $\bar{\partial}_{A}(\Phi):=\eta_{-}(\Phi)+[A_{-1},\Phi]$.
Therefore using that $\Phi^{*}=-\Phi$ we have
\begin{align*}
\Re\langle u_{m-1},\mu_{-}(\Phi\,u_{m})\rangle&=\Re\langle u_{m-1},-\Phi\mu_{+}(u_{m-2})-\Phi^{2}\,u_{m-1}+\bar{\partial}_{A}(\Phi)u_{m}\rangle\\
&=\Re\langle \Phi\,u_{m-1},\mu_{+}(u_{m-2})\rangle+|\Phi\,u_{m-1}|^{2}+\Re\langle u_{m-1},\bar{\partial}_{A}(\Phi)u_{m}\rangle\\
&=-\Re\langle u_{m-2},\mu_{-}(\Phi\,u_{m-1})\rangle+|\Phi\,u_{m-1}|^{2}+\Re\langle u_{m-1},\bar{\partial}_{A}(\Phi)u_{m}\rangle
\end{align*}
and thus
\[\Re\langle u_{m-1},\mu_{-}(\Phi\,u_{m})\rangle+\Re\langle u_{m-2},\mu_{-}(\Phi\,u_{m-1})\rangle=|\Phi\,u_{m-1}|^{2}+\Re\langle u_{m-1},\bar{\partial}_{A}(\Phi)u_{m}\rangle.\]
If we now set
\[a_{m}:=|\mu_{+}(u_{m})|^{2}+|\mu_{+}(u_{m-1})|^{2}\]
then we derive ($m\geq \ell+1$)
\begin{align*}
a_{m+1}&\geq a_{m-1}+|\Phi\,u_{m}|^{2}-|\Phi\,u_{m-1}|^{2}+c_{m+1}|u_{m+1}|^{2}+
c_{m}|u_{m}|^{2}-2\Re\langle u_{m-1},\bar{\partial}_{A}(\Phi)u_{m}\rangle\\
&\geq a_{m-1}+|\Phi\,u_{m}|^{2}-|\Phi\,u_{m-1}|^{2}+c_{m+1}|u_{m+1}|^{2}+
c_{m}|u_{m}|^{2}-|u_{m-1}|^{2}-|\bar{\partial}_{A}(\Phi)u_{m}|^{2}\\
&\geq a_{m-1}-|\Phi\,u_{m-1}|^{2}+c_{m+1}|u_{m+1}|^{2}+
c_{m}|u_{m}|^{2}-|u_{m-1}|^{2}-|\bar{\partial}_{A}(\Phi)u_{m}|^{2}.
\end{align*}
Since $M$ is compact there exist positive constants $B$ and $C$ such that
\begin{align*}
&|\Phi\,f|^{2}\leq (B-1)|f|^{2}\\
&|\bar{\partial}_{A}(\Phi)\,f|^{2}\leq C|f|^{2}
\end{align*}
for any $f\in C^{\infty}(SM,\M_{n}(\C))$. Therefore
\[a_{m+1}\geq a_{m-1}+r_{m}\]
where
\[r_{m}:=-B|u_{m-1}|^{2}+c_{m+1}|u_{m+1}|^{2}+(c_{m}-C)|u_{m}|^{2}.\]
Now choose a positive integer $N_{0}\geq \ell$ large enough so that for $m\geq N_{0}$ we have
\[c_{m}>\max\{B,C\}.\]
Let $m=N+1+2k$, where $k$ is a non-negative integer and $N$ is an integer with $N\geq N_0$. Note that from the definition
of $r_m$ and our choice of $N$ we have
\[r_{m}+r_{m-2}+\cdots+r_{N+1}\geq -B|u_{N}|^2.\]
Thus
\[a_{m+1}\geq a_{N}+r_{m}+r_{m-2}+\cdots+r_{N+1}\geq a_{N}-B|u_{N}|^2.\]
From the definition of $a_{m}$ and (\ref{eq:exp}) we know that
$a_{N}\geq c_{N}|u_{N}|^{2}$ and hence
\[a_{m+1}\geq (c_{N}-B)|u_{N}|^{2}.\]
Since the function $u$ is smooth, $\mu_{+}(u_m)$ must tend to zero
in the $L_{2}$-topology as $m\to\infty$. Hence $a_{m+1}\to 0$
as $k\to\infty$ which in turns implies that $u_{N}=0$ for any $N\geq N_0$.
A similar argument shows that $u_{m}=0$ for all $m$ sufficiently large and negative thus concluding that
$u$ has finite degree as desired.

\end{proof}

An inspection of the proof above gives the following:

\begin{Corollary} Let $M$ be a closed oriented surface of negative curvature
and let $(A,\Phi)$ be a transparent pair, where $A$ is a flat connection. Then $\Phi\equiv 0$ and $A$ is gauge equivalent to the trivial connection.
\end{Corollary}

\begin{proof} Indeed, if $A$ is flat and $K<0$, the equality
\[|\mu_{+}u|^2=|\mu_{-}u|^2+\frac{i}{2}(\langle K\,V(u),u\rangle),\]
implies that $\mu_{+}$ is injective on $\Omega_{n}$ for $n\geq 1$ (and $\mu_{-}$ is injective on $\Omega_{n}$ for $n\leq -1$).
Any $u:SM\to U(n)$ solving $X(u)+(A+\Phi)u=0$ must have a finite Fourier series
and thus if we write $u=\sum_{-N}^{N}u_{m}$, then $\mu_{+}(u_{N})=\mu_{-}(u_{-N})=0$ which in turn implies $u_{N}=u_{-N}=0$ if $N\geq 1$. Arguing inductively, it follows that $u=u_{0}$ and $X(u)+(A+\Phi)u=0$ may be rewritten as $du_{0}+(A+\Phi)u_{0}=0$. This clearly implies $\Phi\equiv 0$ and $A$
gauge equivalent to the trivial connection.

\end{proof}

\section{A general correspondence for pairs}

The purpose of this section is to describe a general classification result for cohomologically
trivial pairs on any surface similar to \cite[Theorem 3.1]{P1}. This correspondence is important for our approach since it will help us expose the relation of the problem at hand with the underlying complex structure of $M$. 
In what follows we assume that $G$ is $SU(n)$ or $SO(n)$ 
with lie algebra $\mathfrak g$.

Let $M$ be an oriented surface with a Riemannian metric and let
$SM$ be its unit tangent bundle. 
Let
\[\cA:=\{A:SM\to \mathfrak{g}:\;\;V^{2}(A)=-A\}.\]
The set $\cA$ is identified naturally (after fixing a metric) with 
$\Omega^{1}(M,\mathfrak{g})$. 
An element in $\Omega^{1}(M,\mathfrak{g})$ is a smooth map
$A:TM\to \mathfrak{g}$ such that for each $x\in M$ it is linear 
in $v\in T_{x}M$ and the bijection with $\cA$ is obtained by restriction
 to $SM$.
To see that this is a bijection note that a function
$A:SM\to\mathfrak{g}$ satisfying $V^{2}(A)+A=0$ is a function that locally (in isothermal coordinates)
can be written as $A(x,y,\theta)=a(x,y)\cos\theta+b(x,y)\sin\theta$ and we recover the element in $\Omega^1(M,\mathfrak{g})$ locally by setting
\[A=A_{x}dx+A_{y}dy\]
where $A_{x}=ae^\lambda$ and $A_{y}=be^\lambda$. It is straightforward to check that this defines a global element in $\Omega^1(M,\mathfrak{g})$.
Under this identification the star operator $\star: \Omega^1(M,\mathfrak{g})\to \Omega^1(M,\mathfrak{g})$ is just $-V:\cA\to\cA$.

A Higgs field $\Phi:M\to\mathfrak{g}$ can also be regarded as a function
$\Phi:SM\to\g$ such that $V(\Phi)=0$ and we denote this set as
$C^\infty_{0}(SM,\g)$.

Recall from the introduction that a pair $(A,\Phi)\in \cA\times C_{0}^{\infty}(SM,\g)$ is said to be cohomologically trivial
if there exists a smooth $u:SM\to G$ such that $C(x,v,t)=u(\phi_{t}(x,v))u^{-1}(x,v)$.
Differentiating with respect to $t$ and setting $t=0$ this is equivalent to
\begin{equation}
X(u)+(A+\Phi)u=0.
\label{eq:trans}
\end{equation}

Let $\P_{0}$ be the set of all cohomologically trivial pairs, that is,
the set of all $(A,\Phi)\in \mathcal A\times C_{0}^{\infty}(SM,\g)$ such that there exists $u:SM\to G$
for which (\ref{eq:trans}) holds.

Given a vector field $W$ in $SM$, let $G_{W}$ be the set
of all $u:SM\to G$ such that $W(u)=0$, i.e. first integrals
of $W$. Note that $G_{V}$ is nothing but the group
of gauge transformations of the trivial bundle $M\times \mathbb F^n$, where
$\mathbb F=\re,\C$.

We wish to understand $\P_{0}/G_{V}$. Now let $\cH_{0}$ be the set
of all pairs $f,\Psi:SM\to\g$ such that
\begin{align*}
H(f)+VX(f)-[X(f),f]+\Psi&=0,\\
V(\Psi)+[f,\Psi]&=0\\
\end{align*}
and there is $u:SM\to G$ such that
$f=u^{-1}V(u)$. It is easy to check that
$G_{X}$ acts on $\cH_{0}$ by 
\[(f,\Psi)\mapsto (a^{-1}f\,a+a^{-1}V(a),\;a^{-1}\Psi a).\]
where $a\in G_{X}$.

\begin{Theorem} There is a 1-1 correspondence between
$\P_{0}/G_{V}$ and $\cH_{0}/G_{X}$.
\label{theorem:1-1}
\end{Theorem}

\begin{proof} Forward direction: a cohomologically trivial pair $(A,\Phi)$
comes with a $u$ such that $X(u)+(A+\Phi)u=0$. Let us set
$f:=u^{-1}V(u)$ and $\Psi:=u^{-1}\Phi u$. We need to check that
$(f,\Psi)\in \cH_{0}$, i.e., the pair $(f,\Psi)$ satisfies the PDEs:
\begin{align}
H(f)+VX(f)-[X(f),f]+\Psi&=0,\label{eq:k1}\\
V(\Psi)+[f,\Psi]&=0 \label{eq:k2}
\end{align}
Using $u$ we may
define a connection on $SM$ gauge equivalent to $\pi^*A$ by setting
$B:=u^{-1}du+u^{-1}\pi^*A u$, where $\pi:SM\to M$ is the foot-point projection.
Using $X(u)+(A+\Phi)u=0$ we derive
\begin{equation}
B(X)=-u^{-1}\Phi u=-\Psi.
\label{eq:b1}
\end{equation}
Note also that $B(V)=f$.
Since $V(\Phi)=0$, the equation $V(u\Psi u^{-1})=0$ gives
\[V(u)\Psi u^{-1}+uV(\Psi)u^{-1}+u\Psi V(u^{-1})=0.\]
Using that $-f=V(u^{-1})u$ we obtain
\[V(\Psi)+[f,\Psi]=0\]
which is (\ref{eq:k2}). To derive (\ref{eq:k1}) we first note that
since $\pi^*A$ is the pull-back of a connection
on $M$, the curvature $F_{B}$ of $B$ must vanish when one of the entries
is the vertical vector field $V$.
Using that $F_{B}=dB+B\wedge B$ and (\ref{eq:b1}) we compute
\[0=F_{B}(X,V)=dB(X,V)+[B(X),B(V)]=dB(X,V)-[\Psi,f].\]
But
\[dB(X,V)=XB(V)-VB(X)-B([X,V])=XB(V)+V(\Psi)+B(H),\]
and combined with (\ref{eq:k2}) this gives
\begin{equation}
B(H)=-XB(V)=-X(f).
\label{eq:cuXV}
\end{equation}
We also compute
\[0=F_{B}(H,V)=dB(H,V)+[B(H),B(V)],\]
and
\[dB(H,V)=HB(V)-VB(H)-B([H,V])=HB(V)-VB(H)-B(X),\]
hence
\begin{equation}
HB(V)-VB(H)+[B(H),B(V)]=B(X).
\label{eq:cuHV}
\end{equation}
Combining (\ref{eq:cuXV}) and (\ref{eq:cuHV}) gives:
\[H(f)+VX(f)-[X(f),f]+\Psi=0\]
which is (\ref{eq:k1}).

Backward direction: Given a pair $(f,\Psi)$ with $fu=V(u)$ satisfying (\ref{eq:k1}) and (\ref{eq:k2}) set $\Phi:=u\Psi u^{-1}$ and 
$A:=-X(u)u^{-1}-\Phi$.
We need to check that $(A,\Phi)\in \cA_{0}\times C^{\infty}_{0}(SM,\g)$, 
i.e. $V^{2}(A)=-A$ and $V(\Phi)=0$. Checking that $V(\Phi)=0$ is easy, simply
use $V(\Psi)+[f,\Psi]=0$ and $fu=V(u)$.
After this, checking that $V^{2}(A)=-A$ is completely analogous to a calculation done in the proof of \cite[Theorem B]{P} and so we omit the details.

Now there are two ambiguities here. Going forward, we may change
$u$ as long as we solve $X(u)+(A+\Phi)u=0$. This changes $(f,\Psi)$ by the action
of $G_{X}$. Going backwards we may change $u$ as long as $fu=V(u)$,
this changes $(A,\Phi)$ by a gauge transformation, i.e. an element
in $G_{V}$.

\end{proof}

\begin{Remark}{\rm 
Note that if the geodesic flow is transitive (i.e. there is a dense orbit)
the only first integrals are the constants and thus $G_{X}=U(n)$
acts simply by conjugation. If $M$ is closed and of negative curvature,
the geodesic flow is Anosov and therefore transitive.

}

\end{Remark}

\begin{Remark}{\rm There is an alternative way of writing equations
(\ref{eq:cuXV}) and (\ref{eq:cuHV}) which reveals a bit of their structure.

A connection on $SM$ is determined as long as we specify the values
of a $\g$-valued 1-form $\tau$ on $SM$. Let $\tau$ be given by
\[\tau(X)=0,\;\;\;\;\tau(H)=-X(f),\;\;\;\;\tau(V)=f.\]
We compute the curvature $F_{\tau}$ of $\tau$ at $(H,V)$ and we find
\[F_{\tau}(H,V)=H(f)+VX(f)-[X(f),f]\]
thus using (\ref{eq:k1}) we derive
\begin{equation}
\Psi=-F_{\tau}(H,V).
\label{eq:k'1}
\end{equation}
Recall that a connection induces a covariant derivative on endomorphisms and for $\tau$ we denote it by $D^{\tau}$. Hence we may write (\ref{eq:k2}) as
\begin{equation} 
D_{V}^{\tau}(\Psi)=0.
\label{eq:k'2}
\end{equation}
Thus equations (\ref{eq:k'1}) and (\ref{eq:k'2}) can be seen as the master equations for pairs. In fact there is obviously just one equation for $f$:
\begin{equation}
D_{V}^{\tau}(F_{\tau}(H,V))=0,
\label{eq:k'3}
\end{equation}
and once $f$ is found we obtain $\Psi$ from (\ref{eq:k'1}).

}

\end{Remark}

\section{The B\"acklund transformation for $SO(3)$-pairs}
\label{bt}
In this section we restrict to the case in which the structure group is
$SO(3)$.

Suppose there is a smooth map $b:SM\to SO(3)$ such that
if we let $f:=b^{-1}V(b)$ and
\[-\Psi:=H(f)+VX(f)-[X(f),f],\]
then $\Psi$ and $f$ are related by the PDE:
\begin{equation}
V(\Psi)+[f,\Psi]=0.
\label{mypde}
\end{equation}
Then, by Theorem \ref{theorem:1-1}, the pair $(f,\Psi)$ determines a cohomologically trivial pair $(A,\Phi)$
with $A+\Phi:=-X(b)b^{-1}$ and $\Phi=b\Psi b^{-1}$.
Using that $[V,X]=H$ we derive:
\begin{align*}
V(A)b=&-VX(b)-X(b)V(b^{-1})b\\
=&-XV(b)-H(b)+X(b)f\\
=&-H(b)-bX(f)\\
\end{align*}
and thus
\begin{equation}
-\star A=V(A)=-bX(f)b^{-1}-H(b)b^{-1}.
\label{eq:useful}
\end{equation}

Now suppose we are given a smooth function $g:M\to\s$ with $g^3+g=0$. By Corollary \ref{cor:existencea} there exists
a smooth $a:M\to SO(3)$ such that $ag=V(a)$. In what follows we shall assume that $g$ is not identically
zero. The equality $g^3+g=0$ means that $g$ has norm one with respect to the canonical inner product 
$\langle \cdot,\cdot\rangle$ in $\s$ defined by $\langle g,h\rangle=\text{trace}(gh^{t})/2$.

Let us set $u:=ab:SM\to SO(3)$, $F:=u^{-1}V(u)=b^{-1}g\,b+f$ and
\[-\Lambda:=H(F)+VX(F)-[X(F),F].\]

\medskip

\noindent{\bf Question.} When does $(F,\Lambda)$ satisfy (\ref{mypde})?

\medskip

If it does, then it defines (via Theorem \ref{theorem:1-1}) a new cohomologically trivial pair given
by
\begin{align*}
A_{g}+\Phi_{g}&=-X(ab)(ab)^{-1}=-X(a)a^{-1}+a(A+\Phi)a^{-1},\\
\Phi_{g}&=(ab)\Lambda (ab)^{-1},\\
\end{align*}
where $(A,\Phi)$ is the cohomologically trivial pair associated to
$(f,\Psi)$. 

Recall that the connection $A$ defines a covariant derivative $d_{A}g=dg+[A,g]$.

\begin{Lemma} $(F,\Lambda)$ satisfies (\ref{mypde}) if and only if
\begin{equation}
-\star d_{A}g=[d_{A}g,g].
\label{gmero}
\end{equation}
\label{lemma:gmero}
\end{Lemma}

\begin{proof}

Starting  with $F=b^{-1}g\,b+f$
and using that $A+\Phi=-X(b)b^{-1}=bX(b^{-1})$ we compute
\[X(F)=b^{-1}\left([A+\Phi,g]+X(g)\right)b+X(f).\]
Similarly, using $H(b)=-(V(A))b-bX(f)$ (cf. (\ref{eq:useful})) we find
\[H(F)=b^{-1}\left([V(A),g]+H(g)\right)b+[X(f), b^{-1}g\,b]+H(f).\]
Now we compute $VX(F)$; here we use that $V(g)=0$.
We obtain
\[VX(F)=[b^{-1}([A+\Phi,g]+X(g))b,f]+b^{-1}\left([V(A),g]+VX(g)\right)b+VX(f).\]
The last term we need in order to compute $\Lambda$ is:

\[[X(F),F]=b^{-1}[[A+\Phi,g]+X(g),g]b+[b^{-1}([A+\Phi,g]+X(g))b,f]+
[X(f),b^{-1}g\,b]+[X(f),f].\]

Note that $X(g)=dg$, $H(g)=-\star dg$ and $V(A)=-\star A$.
Putting everything together, using the definition of $\Psi$ in terms of $f$ and
simplifying we obtain
\[-\Lambda=b^{-1}\left(-[[A+\Phi,g]+dg],g]-2[\star A,g]-2\star dg\right)b-\Psi.\]
We can simplify this further as
\[-\Lambda=b^{-1}\left(-[d_{A}g,g]-[[\Phi,g],g]-2\star d_{A}g-\Phi\right)b.\]
Let us set
\[T:=[d_{A}g,g]+[[\Phi,g],g]+2\star d_{A}g+\Phi,\]
then we see that $(F,\Lambda)$ satisfies (\ref{mypde}) iff $T$ satisfies
\[V(T)+[g,T]=0.\]
Recall that the Lie bracket in $\s$ satisfies $[a,[b,c]]=b\langle a,c\rangle-c\langle a,b\rangle$.
Thus since $g$ has norm one (we are assuming $g^3+g=0$ with $g$ non-zero) we see that
\[[[\Phi,g],g]=-\Phi+g\langle g,\Phi\rangle,\]
and since $\langle d_{A}g,g\rangle=0$ we also see that
\[[g,[d_{A}g,g]]=d_{A}g.\]
Therefore
\begin{align*}
V(T)&=-[\star d_{A}g,g]+2d_{A}g\\
[g,T]&=d_{A}g+2[g,\star d_{A}g].
\end{align*}
Thus $(F,\Lambda)$ satisfies (\ref{mypde}) iff $g$ satisfies:
\[-[\star d_{A}g,g]+d_{A}g=0.\]
and applying $\star$ we see that the last equation is equivalent to (\ref{gmero}).

\end{proof}

From the proof above we can derive a fairly explicit form for $\Phi_{g}$ and $A_{g}$.
Since $\Phi_{g}=(ab)\Lambda(ab)^{-1}$ and
\[b\Lambda b^{-1}=[d_{A}g,g]+g\langle g,\Phi\rangle+2\star d_{A}g\]
we obtain:
\[\Phi_{g}=a([d_{A}g,g]+g\langle g,\Phi\rangle+2\star d_{A}g)a^{-1},\]
and using that $[d_{A}g,g]=-\star d_{A}g$ we have
\begin{equation}
\Phi_{g}=a(g\langle g,\Phi\rangle+\star d_{A}g)a^{-1}.
\label{eq:phi}
\end{equation}
And from this and $A_{g}+\Phi_{g}=-X(ab)(ab)^{-1}=-X(a)a^{-1}+a(A+\Phi)a^{-1}$ we derive the following formula for $A_{g}$:
\begin{equation}
A_{g}=-X(a)a^{-1}+a(A+\Phi-g\langle g,\Phi\rangle-\star d_{A}g)a^{-1}.
\label{eq:aF}
\end{equation}

\subsection{New features} In this subsection we explain the relationship
between the transformation we just introduced and the one described in
\cite{P1} for cohomologically trivial connections (no Higgs field present).
We also explain how to obtain cohomologically trivial $SU(2)$-pairs.

\begin{Lemma} Let $q:=aga^{-1}$. Then $V(q)=0$ and $d_{A_{g}}q=a[\Phi,g]a^{-1}=[a\Phi a^{-1},q]$. Moreover
\[d_{A_{g}}q=-[\star d_{A_{g}}q,q].\]
\end {Lemma} 

\begin{proof} Since $V(g)=0$ we have $V(q)=V(a)ga^{-1}+agV(a^{-1})$.
But $ag=V(a)$ and taking transposes $-ga^{-1}=V(a^{-1})$ (recall that $g^t=-g$ and $a^{-1}=a^t$). Therefore $V(q)=ag^2 a^{-1}-ag^{2}a^{-1}=0$.

Since $V(q)=0$ we may identify $dq$ with $X(q)$, so we compute using (\ref{eq:aF}):
\begin{align*}
dq&=X(q)=X(a)g a^{-1}+aX(g)a^{-1}+agX(a^{-1})\\
&=X(a)a^{-1}q+a\,dg\,a^{-1}+q\,aX(a^{-1})\\
&=[X(a)a^{-1},q]+a\,dg\,a^{-1}\\
&=[-A_{g},q]+a[A+\Phi-\star d_{A}g,g]a^{-1}+a\,dg\,a^{-1}\\
&=[-A_{g},q]+a(d_{A}g-[\star d_{A}g,g])a^{-1}+a[\Phi,g]a^{-1}\\
\end{align*}
and the first part of the lemma follows from the fact that $d_{A}g-[\star d_{A}g,g]=0$.

To prove the equation displayed in the statement of the lemma we compute
using that $d_{A_{g}}q=a[\Phi,g]a^{-1}$ and $ag=V(a)$:
\begin{align*}
-\star d_{A_{g}}q&=V(d_{A_{F}}q)=V(a)[\Phi,g]a^{-1}+a[\Phi,g]V(a^{-1})\\
&=qa[\Phi,g]a^{-1}-a[\Phi,g]a^{-1}q\\
&=[q,[a\Phi a^{-1},q]]\\
&=a\Phi a^{-1}-q\langle q,a\Phi a^{-1}\rangle,\\
\end{align*}
where in the last equation we used that $|q|=1$ and that $[a,[b,c]]=b\langle a,c\rangle-c\langle a,b\rangle$.
Thus
\[-[\star d_{A_{g}}q,q]=[a\Phi a^{-1},q]=d_{A_{g}}q\]
as desired.

\end{proof}

It follows from the lemma that if we let $g':=-q=-aga^{-1}$, then $g'$ satisfies equation (\ref{gmero}). Also, since $a^{-1}g'=-ga^{-1}=V(a^{-1})$ it follows
that if we run the B\"acklund transformation on $(A_{g},\Phi_{g})$ using
$g'$ and $a^{-1}$ we get back to the pair $(A,\Phi)$.

There is an interesting case which arises from the lemma.
Suppose we start with a cohomologically trivial pair of the form $(A,0)$, i.e. $\Phi=0$ and we run the
B\"acklund transformation with $(g,a)$. Using (\ref{eq:phi}) we see that
\[\Phi_{g}=a(\star d_{A}g)a^{-1}.\]
But if $\Phi=0$, then $d_{A_{g}}q=0$ and thus $q$ satisfies (\ref{gmero}).
Choose $b:SM\to SO(3)$ such that $bq=V(b)$ and run the B\"acklund transformation again on the pair $(A_{g},\Phi_{g})$ using $(q,b)$. We obtain a new
cohomologically trivial pair $(A_{q},\Phi_{q})$. Using (\ref{eq:phi}) we see that
\[\Phi_{q}=b(q\langle q,\Phi_{g}\rangle+\star d_{A_{g}}q)b^{-1}.\]
But $d_{A_{g}}q=0$ and
\[\langle q,\Phi_{g}\rangle=\langle aga^{-1},a(\star d_{A}g)a^{-1}\rangle
=\langle g,\star d_{A}g\rangle=0\]
and hence $\Phi_{q}=0$.
Thus doing this special 2-step B\"acklund transformation on the cohomologically trivial connection $A$ produces a new cohomologically trivial connection $A_{q}$.

We claim that this 2-step $SO(3)$ B\"acklund transformation coincides with the
$SU(2)$ B\"acklund transformation for connections introduced
in \cite{P1}. In this way we have introduced a ``square root'' which requires
the intermediate step to have a non-trivial Higgs field.

Here is a way to see this. The 2-step process is implemented by
$c:=a_{q}a$ so let us compute
\[c^{-1}V(c)=a^{-1}a_{q}^{-1}(V(a_{q})a+a_{q}V(a))=a^{-1}qa+g=2g.\]
Consider the isomorphism
$\ell:\s\to \mathfrak{su}(2)$ given by
\[\left(\begin{array}{ccc}
0&t&x\\
-t&0&y\\
-x&-y&0\end{array}\right)\mapsto\frac{1}{2}\left(\begin{array}{cc}
-it&-x-iy\\
x-iy&it\end{array}\right).\]
Given $g:M\to\s$ with norm one, it is easy to check that the map
$h:M\to\mathfrak{su}(2)$ defined as $h:=\ell\circ 2g$ has the property
that $h^{2}=-\mbox{\rm Id}$. Clearly we can reverse this process and obtain
$g$ given $h$.
Now observe that $g$ satisfies $-\star d_{A}g=[d_{A}g,g]$
if and only if $h$ satisfies $-2\star d_{\ell\circ A}h=[d_{\ell\circ A}h,h]$.
The last equation is precisely what is needed to run the $SU(2)$ B\"acklund transformation
on cohomologically trivial connections (i.e. with $\Phi=0$).

We conclude this subsection with the following remark.
Note that a cohomologically trivial $SO(3)$-pair defines
a cohomologically trivial $SU(2)$-pair if and only if there is a trivializing
function $u:SM\to SO(3)$ such that the induced homomorphism
$u_{*}:\pi_{1}(SM)\to \pi_{1}(SO(3))=\Z_{2}$ is trivial. Recall that multiplying a trivializing function on the left by a gauge transformation of $M$ gives a trivializing function of a gauge equivalent pair.
 We claim that
if we start
with $u$ such that $u_{*}=0$ (e.g. we start with $(A,\Phi)=0$) and we apply
the SO(3) B\"acklund transformation an even number of times we obtain a
cohomologically trivial $SU(2)$-pair. Observe first that 
\[[0,2\pi]\ni\theta\mapsto \mbox{\rm Id}+g^2+\sin\theta\,g-\cos\theta\,g^{2}\]
where $g$ has norm one is an explicit non-trivial loop in $SO(3)$, thus
the map $a$ with $g=a^{-1}V(a)$ is such that $a_{*}$ is non-zero.
In fact since $(ab)_{*}=a_{*}+b_{*}\,(\mbox{\rm mod}\,2)$ we see that
after applying two B\"acklund transformations with $a$ and $b$, the
morphism $(ab)_{*}$ has the element generated by the fibres of $SM$ in its kernel
and thus it induces a morphism $\rho:\pi_{1}(M)\to \Z_{2}$. But it is always possible to choose a smooth gauge map $r:M\to SO(3)$ such that $r_{*}=\rho$. Indeed, 
express the surface $M$ as a wedge of circles to which a 2-cell
is attached and use for example \cite[Lemma 4.31]{H} to obtain a continuous
$r:M\to SO(3)$ with $r_{*}=\rho$. Now approximate the continuous map by a smooth one. Hence there is always a smooth
$r:M\to SO(3)$ such that $(rab)_{*}=0$. Hence after applying an even number of $SO(3)$ B\"acklund transformations we obtain a trivializing function which lifts to $SU(2)$.

\subsection{Holomorphic interpretation of (\ref{gmero})}

We will now rephrase equation (\ref{gmero}) in terms of holomorphic line bundles. 
There are two ways in which we can do this. First consider as above the isomorphism $\ell:\s\to\mathfrak{su}(2)$ and
$h:M\to\mathfrak{su}(2)$ defined as $h:=\ell\circ 2g$ satisfying $h^{2}=-\mbox{\rm Id}$.
We pointed out already that $g$ satisfies $-\star d_{A}g=[d_{A}g,g]$
if and only if $h$ satisfies $-2\star d_{\ell\circ A}h=[d_{\ell\circ A}h,h]$.
Now, the meaning of this last equation was analysed in detail in \cite{P1}.
If we let $L$ denote the line bundle given by the eigenspace associated with the eigenvalue $i$
of $h$, it turns out (cf. \cite[Lemma 4.3]{P1}) that $h$ satisfies
$-2\star d_{\ell\circ A}h=[d_{\ell\circ A}h,h]$ if and only if the line bundle
$L$ is holomorphic with respect to the complex structure in $M\times\C^2$ induced by $\ell\circ A$ (we will prove something quite similar below).
Thus maps $g:M\to\s$ with norm one satisfying (\ref{gmero}) are in 1-1 correspondence with holomorphic line subbundles of $M\times\C^2$ with respect to the complex structure induced by $\ell\circ A$. 
It is well known that holomorphic line subbundles always exist (there are always non-zero meromorphic sections).

The second way to interpret (\ref{gmero}) is closely related and is needed to complete the classification in the next section.
For this we consider the monomorphism $SO(3)\hookrightarrow SU(3)$ and the corresponding monomorphism of Lie algebras $\s\hookrightarrow \mathfrak{su}(3)$.
We will thus think of our $SO(3)$-connection also as an $SU(3)$-connection.

Recall that an $SU(3)$-connection $A$ induces a holomorphic structure on the trivial bundle $M\times \C^3$
and on the endomorphism bundle $M\times \M_{3}(\C)$. 
We have an operator $\bar{\partial}_{A}=(d_{A}-i\star d_{A})/2=\bar{\partial}+[A_{-1},\cdot]$
acting on sections $f:M\to \M_{3}(\C)$.
 
Set $\pi:=-g(g+i\mbox{\rm Id})/2$ and $\pi^{\perp}=\mbox{\rm Id}+g(g+i\mbox{\rm Id})/2$ so that
$\pi+\pi^{\perp}=\mbox{\rm Id}$. For each $x\in M$, the map $\pi(x)$ is the Hermitian orthogonal projection over the 1-dimensional subspace $E_{i}(x)$ given by the eigenvectors of $g(x)$ with eigenvalue $i$. The map $\pi^{\perp}(x)$ is the Hermitian orthogonal projection 
onto $\mbox{\rm Ker}(g(x))\oplus E_{-i}(x)$.

\begin{Lemma} Let $g:M\to \s$ be a smooth map with $\langle g(x),g(x)\rangle=1$
for all $x\in M$.
The following are equivalent:
\begin{enumerate}
\item $-\star d_{A}g=[d_{A}g,g]$;
\item $\bar{\partial}_{A}g=i[\bar{\partial}_{A}g,g]$;
\item $E_{i}$ is a $\bar{\partial}_{A}$-holomorphic line bundle;
\item $\pi^{\perp}\bar{\partial}_{A}\pi=0$.
\end{enumerate}
\label{lemma:eq}
\end{Lemma}

\begin{proof}

Suppose that (1) holds. Apply $\star$ to obtain: $d_{A}g=[\star d_{A}g,g]$.
Thus $d_{A}g-i\star d_{A}g=i[d_{A}g-i\star d_{A}g,g]$.
In other words $\bar{\partial}_{A}g=i[\bar{\partial}_{A}g,g]$ which is (2).
Conversely, if (2) holds we recover (1) just by taking real and imaginary 
parts.

Let us show that (1) and (4) are equivalent. We make some preliminary observations. Since $g^3+g=0$ we derive
\[(d_{A}g)g^2+d_{A}g+g(d_{A}g)g+g^{2}(d_{A}g)=0\]
which together with $[g,[d_{A}g,g]]=d_{A}g$ shows that $g(d_{A}g)g=0$ and
\[-d_{A}g=(d_{A}g)g^2+g^{2}(d_{A}g).\]
Consider the equation $\pi^{\perp}\partial_{A}\pi=0$. It is equivalent to
saying that that the image of $\bar{\partial}_{A}\pi$ is contained in
$\mbox{\rm ker}\,\pi^{\perp}=\mbox{\rm Im}\,\pi$. Since the image of $\pi$
is $E_{i}$ the equation $\pi^{\perp}\partial_{A}\pi=0$ is equivalent to
\[g(\bar{\partial}_{A}\pi)=i\,\bar{\partial}_{A}\pi.\]
But
\[\bar{\partial}_{A}\pi=-\frac{1}{2}((\bar{\partial}_{A}g)g+g(\bar{\partial}_{A}g)+i\,\bar{\partial}_{A}g)\]
and thus (4) is equivalent to
\[g(\bar{\partial}_{A}g)g+g^{2}(\bar{\partial}_{A}g)+\bar{\partial}_{A}g=i\,(\bar{\partial}_{A}g)g.\]
Using the observations above this can be further simplified to
\begin{equation}
i\,(\bar{\partial}_{A}g)g=-(\bar{\partial}_{A}g)g^{2}.
\label{eq:cmero}
\end{equation}
We claim that (\ref{eq:cmero}) is equivalent to (1). First note that
(\ref{eq:cmero}) is equivalent to 
\begin{equation}
-(\star d_{A}g)g=(d_{A}g)g^{2}.
\label{eq:cmero1}
\end{equation}
Clearly one can go from (1) to (\ref{eq:cmero1}) by multiplying (1) by $g^2$
and using that $g^3+g=0$ and $g(d_{A}g)g=0$. To go from (\ref{eq:cmero1}) to (1) take transposes in (\ref{eq:cmero1}) to obtain
\[g(\star d_{A}g)=g^{2}(d_{A}g)\]
and if we add this to (\ref{eq:cmero1}) we get
\[-d_{A}g=(d_{A}g)g^2+g^2(d_{A}g)=-[\star d_{A}g,g]\]
which is (1).

Finally we will prove that (3) and (4) are equivalent.
Using the condition $\pi^2=\pi$, we see that $\bar{\partial}_{A}\pi=(\bar{\partial}_{A}\pi)\pi+\pi(\bar{\partial}_{A}\pi)$. Recall that
$\pi^{\perp}\bar{\partial}_{A}\pi=0$ is equivalent
to saying that the image of $\bar{\partial}_{A}\pi$ is contained in
$\mbox{\rm ker}\,\pi^{\perp}=\mbox{\rm Im}\,\pi$ which in turn is equivalent to $\pi(\bar{\partial}_{A}\pi)=\bar{\partial}_{A}\pi$. Hence $\pi^{\perp}\bar{\partial}_{A}\pi=0$ if and only if 
$(\bar{\partial}_{A}\pi)\pi=0$.

The line bundle $E_{i}$ is holomorphic iff given
a local section $\xi$ of $E_{i}$, we have $\bar{\partial}_{A}\xi\in E_{i}$. Applying $\bar{\partial}_{A}$ to $\pi\xi=\xi$ we see that
$\bar{\partial}_{A}\xi\in E_{i}$ iff $(\bar{\partial}_{A}\pi)\xi=0$.
Clearly, this happens iff $(\bar{\partial}_{A}\pi)\pi=0$ and thus (3) holds iff (4) holds.

\end{proof}

The next theorem summarises the B\"acklund transformation introduced in this section and it follows directly from Lemma \ref{lemma:gmero} and Theorem \ref{theorem:1-1}.

\begin{Theorem} Let $(A,\Phi)$ be a cohomologically trivial pair and let
$g:M\to \s$ be a smooth map with $|g(x)|=1$ for all $x\in M$ such that
$-\star d_{A}g=[d_{A}g,g]$.
Consider $a:SM\to SO(3)$ with $g=a^{-1}V(a)$ as given by Corollary \ref{cor:existencea}. Then
\begin{align*}
\Phi_{g}&=a(g\langle g,\Phi\rangle+\star d_{A}g)a^{-1},\\
A_{g}&=-X(a)a^{-1}+a(A+\Phi-g\langle g,\Phi\rangle-\star d_{A}g)a^{-1}\\
\end{align*}
defines a cohomologically trivial pair $(A_{g},\Phi_{g})$.
\label{cor:back}
\end{Theorem}

\section{The classification result}

Let $(A,\Phi)$ be a cohomologically trivial pair with a trivializing function of degree $N\neq 0$. We will show that $(A,\Phi)$ can be obtained as
a B\"acklund transformation of a cohomologically trivial pair which admits a trivializing function of degree $N-1$. Thus by induction we end up showing that all cohomologically trivial pairs which admit a trivializing function with finite Fourier series are B\"acklund transformations of cohomologically trivial pairs of degree zero, i.e. those pairs $(A,\Phi)$ with $\Phi=0$ and $A$ gauge equivalent to the
trivial connection. This will be our classification result and its consequence for negative curvature is Theorem \ref{theorem:main} below.

Let $(A,\Phi)$ be a cohomologically trivial pair with $A+\Phi=-X(b)b^{-1}$ and
$f=b^{-1}V(b)$, where $b:SM\to SO(3)$ is a trivializing function.
We think of $b$ as a map $b:SM\to SO(3)\hookrightarrow SU(3)\hookrightarrow  \M_{3}(\C)$ and as such we make the following:

\medskip

\noindent{\bf Assumption.} Suppose $b$ has a finite Fourier
expansion, i.e., $b=\sum_{k=-N}^{k=N}b_{k}$, where $N\geq 1$.
By Theorem \ref{theorem:finite} we know that this holds if $M$ has negative curvature.

\medskip
Let us assume also that $N$ is the degree of $b$ and thus both $b_{N}$ and 
$b_{-N}=\bar{b}_{N}$ are not identically zero.

The orthogonality condition $bb^t=b^t b=\mbox{\rm Id}$ implies that 
$b_{N}b^{t}_{N}=b^{t}_{N}b_{N}=0$. 
Since $b_{N}$ and $b_{N}^{t}$ must have the same rank,
$\mbox{\rm Im}\,b_{N}\subset \mbox{\rm Ker}\,b_{N}^{t}$ and
$\mbox{\rm Im}\,b_{N}^{t}\subset \mbox{\rm Ker}\,b_{N}$, it follows
that $\dim\,\mbox{\rm Im}\,b_{N}\leq 1$ and $\dim\,\mbox{\rm Ker}\,b_{N}\geq 2$
with a similar statement holding for $b_{N}^{t}$.
Since $b_N$ is not identically zero it follows that the rank of $b_N$ is
one on an open set, which, as we will see shortly, must be all of $M$ except for perhaps
a finite number of points.

 Consider now a fixed vector $\xi\in \mathbb C^3$ such
that $s(x,v):=b_{-N}(x,v)\xi\in \mathbb C^3$ is not zero identically.
Clearly $s$ can be seen as a section of $(M\times \mathbb C^3)\otimes K^{\otimes -N}$.
We may write $b_{-N}$ in local isothermal coordinates as $b_{-N}=he^{-iN\theta}$,
using the notation from
Section \ref{prelim}.
We can thus write $s$ locally as $s=e^{N\lambda}h\xi(d\bar{z})^N$.

\begin{Lemma} The local section $e^{-2N\lambda}s$ is $\bar{\partial}_{A}$-holomorphic.
\end{Lemma}

\begin{proof}Using the operators $\mu_{\pm}$ introduced in Section \ref{prelim} we can write $X(b)+(A+\Phi)b=0$ as
\[\mu_{+}(b_{k-1})+\mu_{-}(b_{k+1})+\Phi\,b_k=0\]
for all $k$. This gives $\mu_{+}(b_{N})=\mu_{-}(b_{-N})=0$.
But $\mu_{-}(b_{-N})=0$ is saying that $e^{-2N\lambda}s$ is $\bar{\partial}_{A}$-holomorphic. 
Indeed, using (\ref{eq:mu}), we see that $\mu_{-}(b_{-N})=0$ implies
\[ \bar{\partial}(he^{-N\lambda})+A_{\bar{z}}he^{-N\lambda}=0\]
which in turn implies
\[ \bar{\partial}(e^{-N\lambda}h\xi)+A_{\bar{z}}e^{-N\lambda}h\xi=0.\]
This equation says that $e^{-2N\lambda}s=e^{-N\lambda}h\xi(d\bar{z})^N$ is $\bar{\partial}_{A}$-holomorphic.

\end{proof}

The section $s$ spans a line bundle $L$ over $M$ which by the previous lemma is 
$\bar{\partial}_{A}$-holomorphic. The section $s$ may have zeros, but at a zero $z_0$, the line bundle extends holomorphically. Indeed, in a neighbourhood
of $z_0$ we may write
$e^{-2N\lambda(z)}s(z)=(z-z_0)^k w(z)$, where $w$ is a local holomorphic
section with $w(z_0)\neq 0$. The section $w$ spans a holomorphic line subbundle which coincides with the one spanned by $s$ off $z_0$.
Therefore $L$ is a $\bar{\partial}_{A}$-holomorphic line bundle that contains the image of $b_{-N}$.
We summarise this in a lemma (recall that we are assuming $N\neq 0$):

\begin{Lemma} The line bundle $L$ determined by the image of $b_{-N}$
is $\bar{\partial}_{A}$-holomorphic.
\label{lemma:holo}
\end{Lemma}

We now wish to use the line bundle $L$ to construct an appropriate $g:M\to\s$
with norm one
such that when we run the B\"acklund transformation from the previous section
we obtain a cohomologically trivial pair of degree $\leq N-1$.

We would also like to find $a:SM\to SO(3)$ such that $u:=ab$ has degree
$\leq N-1$. The map $a$ should be related to $g$ by $ag=V(a)$ and
$a=a_{-1}+a_{0}+a_{1}$. For $u=ab$ to have degree $\leq N-1$ we need
\begin{align}
a_{1}b_{N-1}+a_{0}b_{N}&=0, \label{eq:degN}\\
a_{1}b_{N}&=0.
\label{eq:degN+1}
\end{align}
Note that by conjugating the previous two equations we obtain the corresponding relations involving $b_{-N}$ which ensure that $u$ has degree $\leq N-1$.

Write $b_{N}=C+iD$, where $C$ and $D$ are $3\times 3$ {\it real} matrices.
Since $\mbox{\rm Ker}\,b_{N}^t$ has (complex) dimension two, there must exist
a real vector $0\neq x\in \re^3\subset \C^3$ such that $C^{t}x=D^{t}x=0$.
The relations $b_{N}b_{N}^t=b_{N}^t b_{N}=0$ imply that for any real vector
$y\in\re^3$ we have:
\begin{align*}
&|C^{t}y|=|D^{t}y|,\\
&|Cy|=|Dy|\\
&\langle Cy,Dy\rangle=0.
\end{align*}
Consider now $y\in\re^3$ such that $Cy\neq 0$. Then $\{x,Cy,Dy\}$
is an orthogonal basis of $\re^3$ such that $|Cy|=|Dy|$.
Define $g\in \s$ with norm one as follows:
\begin{align*}
&g(x)=0,\\
&g(Cy)=Dy,\\
&g(Dy)=-Cy.
\end{align*}
Obviously the eigenspace $E_{i}$ of $g$ corresponding to the eigenvalue
$i$ is spanned by $Cy-iDy$ and $E_{-i}$ is spanned by $Cy+iDy$. Thus
$\mbox{\rm Im}\,b_{N}=E_{-i}$ and $\mbox{\rm Ker}\,b_{N}^{t}=E_{-i}\oplus\mbox{\rm Ker}(g)$ and $L=E_{i}=\mbox{\rm Im}\,b_{-N}$.

For this choice of $g$, Corollary \ref{cor:existencea} gives a smooth
$a:SM\to SO(3)$ with $ag=V(g)$ and such that
\begin{align*}
&\mbox{\rm Ker}\,a_{0}=E_{i}\oplus E_{-i},\\
&\mbox{\rm Ker}\,a_{1}=\mbox{\rm Ker}(g)\oplus E_{-i},\\
&\mbox{\rm Ker}\,a_{-1}=\mbox{\rm Ker}(g)\oplus E_{i}.
\end{align*}
Hence $\mbox{\rm Im}\,b_{N}\subset \mbox{\rm Ker}\,a_{1}$ and
$\mbox{\rm Im}\,b_{N}\subset \mbox{\rm Ker}\,a_{0}$ and thus 
$a_{1}b_{N}=a_{0}b_{N}=0$. This gives (\ref{eq:degN+1}) and to get
(\ref{eq:degN}) we need to show that $a_{1}b_{N-1}=0$.

Since $b\in SO(3)$, $b_{N}^{t}b_{N-1}+b_{N-1}^{t}b_{N}=0$ which says that
the complex matrix $b_{N}^{t}b_{N-1}$ is antisymmetric. Since the kernel
of $b_{N}$ is two dimensional, it follows that the kernel of $b_{N}^{t}b_{N-1}$
is at least two dimensional which combined with the fact that it is
antisymmetric forces $b_{N}^{t}b_{N-1}=0$. In other words
$\mbox{\rm Im}\,b_{N-1}\subset \mbox{\rm Ker}\,b_{N}^{t}=\mbox{\rm Ker}\,a_{1}$
which gives $a_{1}b_{N-1}=0$. 

Finally, Lemmas \ref{lemma:holo} and \ref{lemma:eq} tell us that
$-\star d_{A}g=[d_{A}g,g]$ and therefore by Theorem \ref{cor:back}, $u$ gives rise to a cohomologically trivial 
pair. Combining this with Theorem \ref{theorem:finite} we have proved:

\begin{Theorem} Let $M$ be a closed orientable surface of negative curvature.
Then any transparent $SO(3)$-pair can be obtained by successive applications
of B\"acklund transformations as described in Theorem \ref{cor:back}.
\label{theorem:main}
\end{Theorem}

\end{document}